\documentclass[10pt,twoside]{article}
\usepackage{amsmath,amssymb,amsthm,mathptmx,hyperref,nicefrac,color}
\usepackage{titlesec}
\usepackage{indentfirst}
\usepackage[inner=2.4cm,outer=2.65cm,bottom=3cm]{geometry}
\titleformat*{\section}{\normalsize\bfseries}
\titleformat*{\subsection}{\normalsize\itshape}
\numberwithin{equation}{section}
\newcommand{\f}[1]{\pmb{#1}}
\DeclareMathOperator{\R}{\mathbb{R}}

\DeclareMathOperator{\C}{\mathcal{C}}

\DeclareMathOperator{\N}{\mathbb{N}}
\DeclareMathOperator{\V}{\f H ^1 _{0,\sigma}}
\DeclareMathOperator{\Vd}{(\f H ^{1} _{0,\sigma})^*}

\DeclareMathOperator{\Ha}{\f L^2_{\sigma}}
\DeclareMathOperator{\Hb}{\f{H}^1_0}
\DeclareMathOperator{\Hc}{\f H^2}

\DeclareMathOperator{\He}{\f{H}^1}
\DeclareMathOperator{\Le}{\f{L}^2}
\DeclareMathOperator{\E}{\mathcal{E}}
\DeclareMathOperator{\W}{\mathcal{W}}

\DeclareMathOperator{\F}{\mathcal{F}}
\DeclareMathOperator{\Lap}{\Delta_{\f \Lambda}}

\DeclareMathOperator{\ra}{\rightarrow}
\DeclareMathOperator{\de}{\text{d}}
\DeclareMathOperator{\tr}{tr}

\newcommand{\pat}[2]{\frac{\partial #1}{\partial #2}}
\newcommand{\va}[1]{\frac{\delta \mathcal{F}}{\delta #1}}
\DeclareMathOperator{\di}{\nabla \cdot}

\DeclareMathOperator{\sym}{{sym}}
\DeclareMathOperator{\skw}{skw}
\newcommand{\sy}[1]{(\nabla \f #1)_{{\sym}}}
\newcommand{\sk}[1]{(\nabla \f #1)_{\skw}}
\renewcommand{\t}{\partial_t  }
\newcommand{\syn}[1]{(\nabla \f {#1}_n)_{\sym}}

\newcommand{\intt}[1]{\int_{0}^t\left({ #1}\right) \de s}
\newcommand{\intte}[1]{\int_{0}^t{ #1} \de s}

\newcommand{\intter}[1]{\int_{0}^T{ #1} \text{\emph{d}} t}
\newcommand{\inttes}[1]{\int_{0}^T{ #1} \text{d} t}
\newcommand{\inttu}[1]{\int_{0}^t\left ({ #1}\right ) \text{\emph{d}} s}

\newcommand{\intet}[1]{\int_{\Omega}{ #1} \de \f x}
\newcommand{\ov}[1]{\overline{#1}}

\newcommand{\vv}{\tilde{\f v}}
\newcommand{\dd}{\tilde{\f d}}
\newcommand{\syv}{(\nabla \tilde{\f v})_{{\sym}}}
\newcommand{\skv}{(\nabla \tilde{\f v})_{\skw}}
\renewcommand{\o}{\otimes}

\newtheorem{theorem}{Theorem}
\numberwithin{theorem}{section}
\numberwithin{equation}{section}
\newtheorem{lem}[theorem]{Lemma}
\newtheorem{rem}[theorem]{Remark}
\newtheorem{defi}[theorem]{Definition}

\author{Etienne Emmrich\thanks{%
        Technische Universit\"{a}t Berlin,
        Institut f\"{u}r Mathematik,
        Stra{\ss}e des 17.~Juni 136,
        10623 Berlin, Germany
        \newline{\tt emmrich@math.tu-berlin.de}
        }
\and    Robert Lasarzik\thanks{%
        Technische Universit\"{a}t Berlin,
        Institut f\"{u}r Mathematik,
        Stra{\ss}e des 17.~Juni 136,
        10623 Berlin, Germany
        \newline{\tt lasarzik@math.tu-berlin.de}
        }
}%
\title{Weak-strong uniqueness for the general Ericksen--Leslie system in three dimensions\footnote{This work was funded by CRC 901 {\em Control of self-organizing nonlinear systems: Theoretical methods and concepts of application} (Project A8)\/.
}}

\begin{document}
\markboth{Weak-strong uniqueness for the Ericksen--Leslie model}{E.~Emmrich and R.~Lasarzik}
\date{Version \today}
\maketitle
%%%%%%%%%%%%%%%%%%%%%%%%%%%%%%%%%%%%%%%%%%%%%
\begin{abstract} We study the Ericksen--Leslie system equipped with a quadratic free energy functional. The norm restriction of the director is incorporated by a standard relaxation technique using a double-well potential. We use the relative energy  concept, often applied in the context of compressible Euler- or related systems of fluid dynamics, to prove weak-strong uniqueness of solutions.
A main novelty is that the relative energy inequality is proved for a system with a nonconvex energy.
\end{abstract}
\setcounter{tocdepth}{1}
\tableofcontents
\section{Introduction\label{sec:int}}
This paper is devoted to the weak-strong uniqueness of weak solutions to the three-dimensional Ericksen--Leslie model describing liquid crystal flow. The Ericksen--Leslie model (proposed by Ericksen~\cite{Erick} and Leslie~\cite{leslie}) is  a very successful model for nematic liquid crystals and agrees with experiments (see~\cite[Sec.~11.1,~p.~463]{beris}).
The particular model strongly depends on the choice of the free energy.

Recently, global existence of weak solutions for a very general class of free energies was shown in~\cite{unsere}.
In this article, we prove weak-strong uniqueness of these solutions for a special, physically relevant free energy.
The weak-strong uniqueness property says that the weak solution coincides with a weak solution admitting additional regularity as long as the latter exists. We use the concept of relative energy (see Feiereisl, Jin and Novotn{\'y}~\cite{Feireislrelative}), which can also be used to consider other problems such as the stability of an equilibrium (see Feireisl~\cite{feireislstab}), singular limits for vanishing coefficients (see Breit, Feireisl and Hofmanova~\cite{breit} or Feireisl~\cite{feireislsingular}) or to derive a posteriori estimates for modeling errors (see Fischer~\cite{fischer}).
In the paper at hand, we generalize the relative energy approach to a model with a nonconvex energy.
\subsection{Review of known results}
A simplified Ericksen--Leslie model

\begin{align}
\begin{split}
 \t \f d + ( \f v \cdot \nabla ) \f d  & =  \Delta \f d
 +\frac{1}{\varepsilon^2}( |\f d |^2 -1 ) \f d \,,\\
\t \f v  + ( \f v \cdot \nabla ) \f v + \nabla p -  \Delta \f d
 & = -  \di \left ( \nabla \f d ^T \nabla \f d \right ),\\
 \di \f v &= 0 \,,
 \end{split}\label{simpleEL1}
\end{align}
was first considered in Lin and Liu~\cite{linliu1}, where global existence of weak solutions as well as local existence of strong solutions was shown. Later, Lin and Liu~\cite{linliu3} showed the same result for a generalized system. Existence of weak solutions to the model considered in the article at hand (see~\eqref{eq:strong} below) equipped with the Dirichlet energy with double-well potential
\begin{align*}
F(\f d ,\nabla \f d ) = \frac{k}{2}| \nabla \f d |^2 +\frac{1}{4\varepsilon}( | \f d |^2-1)^2\,.
\end{align*}
 was first proved in {Cavaterra}, {Rocca and  Wu}~\cite{allgemein}.
In~\cite{unsere} the existence of weak solutions to the model considered in the article at hand was proved for a more general class of free-energies.

The concept of weak-strong uniqueness was first considered by Prodi in 1959~(see~\cite{prodi}) and Serrin in 1962 (see~\cite{serrin}). Both studied the Navier--Stokes equation and showed weak-strong uniqueness for a class of weak solutions fulfilling  additional regularity requirements.

There is several work on the weak-strong uniqueness property for different simplifications of the Ericksen--Leslie model. Zhao and Liu~\cite{zhao} established weak-strong uniqueness for the simplified system~\eqref{simpleEL1}
 with different assumptions on the strong solution. Dai~\cite{dai2,dai1} established weak-strong uniqueness for a simplified incompressible model and a more general incompressible Ericksen--Leslie model with additional assumptions on the weak solution, which cannot be shown to hold in general.
Yang et\,al.~\cite{yang} showed the weak-strong uniqueness for the incompressible simplified Ericksen--Leslie system with no nonlinear penalization using ideas of Feireisl et\,al.~\cite{suitable} based upon relative entropy and suitable weak solutions. In the article at hand, we use similar ideas. However, we are able to incorporate the nonlinear part of the free energy in the relative entropy, which we call relative energy, and to show the weak-strong uniqueness without further assumptions on the weak solution.
This is done by adapting the relative energy to the nonconvex energy of the system.
A similar weak-strong uniqueness result for measure-valued solutions to the Ericksen--Leslie system equipped with the nonconvex Oseen--Frank energy (see~\cite{masse} for the existence of such solutions) was recently proved by the second author~\cite{weakstrong}.

\subsection{Notation\label{sec:not}}
%%%%%%%%%%%%%%%%%%%%%%%%%%%%%%%%%
%%%%%%%%%%%Notation Tensoren
Vectors of $\R^3$ are denoted by bold small Latin letters. Matrices of $\R^{3\times 3}$ are denoted by bold capital Latin letters.
Moreover, numbers are denoted be small Latin or Greek letters, and capital Latin letters are reserved for potentials.

%%%%%%%%%%% Notation Skalarptodukt und innner Produkte
The Euclidean inner product in $\R^3$ is denoted by a dot,
$ \f a \cdot \f b : = \f a ^T \f b = \sum_{i=1}^3 \f a_i \f b_i$  for $ \f a, \f b \in \R^3$.
The Frobenius inner product in the space $\R^{3\times 3}$ of matrices is denoted by a double dot, $ \f A: \f B:= \tr ( \f A^T \f B)= \sum_{i,j=1}^3 \f A_{ij} \f B_{ij}$ for $\f A , \f B \in \R^{3\times 3}$.
We also employ the corresponding Euclidean norm with $| \f a|^2 = \f a \cdot \f a$ for $ \f a \in \R^3$ and Frobenius norm with $ |\f A|^2=\f A:\f A$ for $\f A \in \R^{3\times 3}$.
The product of a fourth order with a second order tensor is defined by
\begin{align*}
\f \Gamma : \f A : =\left [ \sum_{k,l=1} ^3 \f \Gamma_{ijkl} \f A_{kl}\right ]_{i,j=1}^3, \quad    \f \Gamma \in \R^{3\times 3 \times 3 \times 3},  \,  \f A \in \R^{3\times 3 }  .
\end{align*}
The standard matrix and matrix-vector multiplication, however, is written without an extra sign for bre\-vi\-ty,
$$\f A \f B =\left [ \sum _{j=1}^3 \f A_{ij}\f B_{jk} \right ]_{i,k=1}^3 \,, \quad  \f A \f a = \left [ \sum _{j=1}^3 \f A_{ij}\f a_j \right ]_{i=1}^3\, , \quad  \f A \in \R^{3\times 3},\,\f B \in \R^{3\times3} ,\, \f a \in \R^3 .$$
%%%%%%%%%%% Vektornotation
The outer product is denoted by
$\f a \otimes \f b = \f a \f b^T = \left [ \f a_i  \f b_j\right ]_{i,j=1}^3$ for $\f a , \f b \in \R^3$. Note that
$\tr (\f a \otimes \f b  ) = \f a\cdot \f b$.
The symmetric and skew-symmetric part of a matrix are denoted by $\f A_{\sym}: = \frac{1}{2} (\f A + \f A^T)$ and
$\f A _{\skw} : = \frac{1}{2}( \f A - \f A^T)$ for $\f A \in \R^{3\times  3}$, respectively.

%%%%%%%%% Nabla operator
We use  the Nabla symbol $\nabla $  for real-valued functions $f : \R^3 \to \R$  as well as vector-valued functions $ \f f : \R^3 \to \R^3$  denoting
\begin{align*}
\nabla f := \left [ \pat{f}{\f x_i} \right ] _{i=1}^3\, ,\quad
\nabla \f f  := \left [ \pat{\f f _i}{ \f x_j} \right ] _{i,j=1}^3 \, .
\end{align*}
For brevity, we write $ \nabla \f f^T $ instead of $ ( \nabla \f f)^T$. The divergence of a vector-valued function $\f f:\R^3\ra \R^3$ and a matrix-valued function $\f A: \R^3 \ra \R^{3\times 3}$ is defined by
\begin{align*}
\di \f f := \sum_{i=1}^3 \pat{\f f _i}{\f x_i} = \tr ( \nabla \f f)\, , \quad  \di \f A := \left [\sum_{j=1}^3 \pat{\f A_{ij}}{\f x_j}\right] _{i=1}^3\, .
\end{align*}
Note that $( \f v\cdot \nabla ) \f f = ( \nabla \f f) \f v = \nabla \f f\, \f v $ for vector-valued functions $\f v$, $\f f: \R^3 \ra \R^3 $.

Throughout this paper, let $\Omega \subset \R^3$ be a bounded domain of class $\C^{3,1}$.
We rely on the usual notation for spaces of continuous functions, Lebesgue and Sobolev spaces. Spaces of vector-valued functions are  emphasized by bold letters, for example
$
\f L^p(\Omega) := L^p(\Omega; \R^3)$,
$\f W^{k,p}(\Omega) := W^{k,p}(\Omega; \R^3)$.
If it is clear from the context, we also use this bold notation for spaces of matrix-valued functions.
For brevity, we often omit calling the domain $\Omega$.
The standard inner product in $L^2 ( \Omega; \R^3)$ is denoted by
$ (\cdot \, , \cdot )$ and in $L^2 ( \Omega ; \R^{3\times 3 })$
by $(\cdot ; \cdot )$.

The space of smooth solenoidal functions with compact support is denoted by $\mathcal{C}_{c,\sigma}^\infty(\Omega;\R^3)$. By $\f L^p_{\sigma}( \Omega) $, $\V(\Omega)$,  and $ \f W^{1,p}_{0,\sigma}( \Omega)$, we denote the closure of $\mathcal{C}_{c,\sigma}^\infty(\Omega;\R^3)$ with respect to the norm of $\f L^p(\Omega) $, $ \f H^1( \Omega) $, and $ \f W^{1,p}(\Omega)$, respectively $(1\leq p<\infty)$.

The dual space of a Banach space $V$ is always denoted by $ V^*$ and equipped with the standard norm; the duality pairing is denoted by $\langle\cdot, \cdot \rangle$. The duality pairing between $\f L^p(\Omega)$ and $\f L^q(\Omega)$ (with $1/p+1/q=1$), however, is denoted by $(\cdot , \cdot )$ or $( \cdot : \cdot )$.

The Banach space of linear bounded operators mapping a Banach space $V$ into itself is denoted by $\mathcal{L}(V)$ and equipped with the usual norm.
For a given Banach space $ V$, Bochner--Lebesgue spaces are denoted, as usual,  by $ L^p(0,T; V)$. Moreover,  $W^{1,p}(0,T; V)$ denotes the Banach space of abstract functions in $ L^p(0,T; V)$ whose weak time derivative exists and is again in $ L^p(0,T; V)$ (see also
Diestel and Uhl~\cite[Section~II.2]{diestel} or
Roub\'i\v{c}ek~\cite[Section~1.5]{roubicek} for more details).
We often omit the time interval $(0,T)$ and the domain $\Omega$ and just write, e.g., $L^p(\f W^{k,p})$ for brevity.
By
%$ \AC ( [0,T]; V)$, $\C([0,T]; V) $, and
 $ \C_w([0,T]; V)$, we denote the spaces of abstract functions mapping $[0,T]$ into $V$ that are
 continuous with respect to the weak topology in $V$.

%%%%%%%%%%%%%%%%%%%%%%%%%%%%%%%%%
%%%%%%%%%%%%%%%%%%%%%%%%%%%%%%%
%%%%%%%%%%%%%%%%%%%%%%%%%%%%5
By $\f \Lambda$, we denote a constant tensor of order 4 that is symmetric, i.e.,\,$\f \Lambda_{ijkl}=\f \Lambda _{klij} $, $ijkl\in\{1,2,3\}$, and obeys the strong ellipticity  condition  (see Mc Lean~\cite{mclean}), i.e.,\,there exists $\eta>0$ such that
\begin{align}
(\f a \otimes \f b ) : \f \Lambda :( \f a \otimes \f b ) \geq \eta | \f a |^2 | \f b|^2  \quad\text{for all }\f a, \f b\in\R^3\, .\label{elast}
\end{align}
We introduce the norm $\|\cdot\|_{\f \Lambda} : = \| \cdot : \f \Lambda :  \cdot\|_{L^1}^{{1}/{2}}$. The norm $\| \nabla \cdot \|_{\f \Lambda}$  is equivalent to the $\He$-norm on $\Hb$.
We use the abbreviation $\Delta_{\f \Lambda}\f d $ for the operator $ \di \f \Lambda : \nabla\f d $ for $\f d \in \Hc$.

Finally, by $c>0$, we denote a generic positive constant and by $C_\delta$ a constant depending on a given parameter  $\delta>0$.

\section{Model and main result\label{sec:mod}}
We consider the general Ericksen--Leslie system, which was investigated in~\cite{unsere}.
In comparison to the model in~\cite{unsere}, we consider a particular free energy function and reformulate the stress tensor. The system is given by
\begin{subequations}\label{eq:strong}
\begin{align}
\t {\f v}  + ( \f v \cdot \nabla ) \f v + \nabla p - \nabla \f d^T \f q - \di  \f T^L&= \f g, \label{nav}\\
\t {\f d }+ ( \f v \cdot \nabla ) \f d -\sk{v}\f d + \lambda \sy{v} \f d +\gamma \f q & =0,\label{dir}\\
\di \f v & = 0\, .
\end{align}%
\end{subequations}%
The vectorfields $\f v: \ov \Omega\times [0,T] \ra \R^3 $ and $\f d: \ov \Omega\times [0,T] \ra \R^3$ represent the velocity field and the director field, respectively. The pressure is denoted by $p: \ov \Omega\times [0,T] \ra \R$. In the article at hand, we do not address the problem of existence or uniqueness of the pressure.
For the free energy potential, we choose the function
\begin{align}
F( \f d , \nabla \f d ):= \frac{1}{2} \nabla \f d :\f \Lambda : \nabla \f d + \frac{1}{4\varepsilon}( | \f d |^2-1)^2\, .
\end{align}
Here $\f \Lambda$ is a constant symmetric fourth order tensor fulfilling the strong ellipticity condition~\eqref{elast} (see Mc Lean~\cite{mclean} and Section~\ref{sec:not}).
Moreover, $\varepsilon >0$ denotes the fixed parameter for the relaxation of the requirement $|\f d|=1$.
We do not address the question of the limit $\varepsilon\ra 0$. For such a singular limit analysis for $\varepsilon \ra 0$ in the context of the Ericksen--Leslie model, we refer to~\cite{weakstrong}.

The free energy is the functional induced by the free energy potential,
\begin{align}
\F(\f d) : = \intet{F( \f d , \nabla \f d)} = \frac{1}{2} \| \nabla \f d\|_{\f \Lambda }^2 + \frac{1}{4\varepsilon }\left \| |\f d|^2-1\right \|_{L^2}^2\, .\label{F}
\end{align}
The vector $\f q$ is the variational derivative of the free energy,
\begin{align}
\f q : = \va{\f d } = -\Delta_{\f \Lambda} \f d + \frac{ 1}{\varepsilon}{ (| \f d |^2 - 1) \f d}\, .\label{q}
\end{align}
For the definition of the operator~$\Delta_{\f \Lambda}$, see Section~\ref{sec:not}.
In comparison to the system studied in~\cite{unsere}, the divergence of the Ericksen stress given by $\di \f T^E= \di \left ( \nabla \f d^T (\partial F/ \partial \nabla \f d) \right ) $ is replaced by $-\nabla \f d ^T \f q$. This reformulation is valid due to the integration-by-parts formula
\begin{align*}
\left ( \di \f T^E , \f \varphi \right ) = - \left ( \nabla \f d^T \f q , \f \varphi \right ) + \left ( \nabla F , \f \varphi \right ) \
\end{align*}
derived in~\cite[Section~3.3]{unsere} that holds for every test function $\f \varphi\in \C_c^\infty(\Omega;\R^3 )$. Hence, via a reformulation, the term $ F$ can be incorporated in the pressure and one ends up with the formulation~\eqref{eq:strong}.
The Leslie stress tensor is given by
\begin{align}
\begin{split}
\f T^L:=\,&  \mu_1 (\f d\cdot \sy  v \f d )\f d \otimes \f d +\mu_4 \sy v
-\gamma (\mu_2+\mu_3) {}\left (\f d \otimes \f q    \right )_{\sym} + \left (\f d \otimes \f q  \right )_{\skw} \\& + ((\mu_5+\mu_6) -\lambda(\mu_2+\mu_3))  \left ( \f d\otimes  \sy v\f d\right )_{\sym} ,
\end{split}\label{leslie}
\end{align}
Note that in view of~\eqref{dir}, the formulation~\eqref{leslie} is equivalent to the formulation of the Leslie stress in~\cite{unsere}.
In order to assure the dissipative character of the system, we assume that the parameters $\lambda$, $\gamma$, $\mu_1$, $\mu_2$, $\mu_3$, $\mu_4$, $\mu_5$, and $\mu_6$ satisfy
%need to stipulate bounds for the occuring constats, i.e.
\begin{align}
\begin{split}
 \mu_1 & > 0, \quad \mu_4 > 0,  \quad  \gamma > 0 ,\quad (\mu_5+\mu_6)- \lambda (\mu_2+\mu_3)>0    \\  4 &\gamma ( (\mu_5+\mu_6)- \lambda (\mu_2+\mu_3))> ( \gamma (\mu_2+\mu_3) -\lambda)^2\,.
 \end{split}\label{con}
\end{align}
Finally, we assume that
$\f g\in    L^2 ( 0,T; (\V)^*)%\label{g}
$.
We equip the system with initial conditions and Dirichlet boundary conditions such that
\begin{subequations}\label{boundary}
\begin{align}
\f v(\f x, 0) &= \f v_0 (\f x) \quad\text{for } \f x \in \Omega ,& \quad \f v (  \f x, t ) &= \f 0 \quad &\text{for }( t,  \f x ) \in [0,T] \times \partial \Omega , \\
\f d (  \f x, 0 ) & = \f d_0 ( \f x) \quad\text{for } \f x \in \Omega , & \quad \f d (  \f x ,t ) & = \f d_1\quad &\text{for }( t,  \f x ) \in [0,T] \times \partial \Omega .
\end{align}
\end{subequations}
We always assume that $\f d_0=\f d_1$ on $\partial \Omega$, which is a compatibility condition providing regularity. For the initial and boundary values, we assume the regularity
\begin{align}
\f v_0 \in \Ha\,,\quad \f d_0 \in \He \, , \text{ and}\quad \f d_1 \in \f H^{3/2}(\partial\Omega)\,.\label{Rand}
\end{align}
\begin{defi}\label{def:weak}
The pair $( \f v ,\f d )$ is said to be a \textit{weak solution} to system~(\ref{eq:strong})--(\ref{Rand}) if
 \begin{align}
\begin{split}
\f v &\in L^\infty(0,T;\Ha )\cap  L^2(0,T;\V ) \cap W^{1,{2}}(0,T; ( \f W^{1,6}_{0,\sigma} )^*),
\\ \f d& \in L^\infty(0,T;\He )\cap  L^2(0,T;\Hc ) \cap W^{1,2}  (0,T;   \f  L^{{3/2}}  ),
\end{split}\label{weakreg}
\end{align}
and
\begin{subequations}\label{weak}
\begin{align}
\begin{split}
\int_0^T \langle\t \f v, \f \varphi\rangle \text{\emph{d}} s + \int_0^T ((\f v\cdot \nabla) \f v, \f \varphi) \text{\emph{d}}s  - \intter{\left \langle\nabla
\f d^T \f q  ,  \f \varphi
\right \rangle}+ \intter{(\f T^L: \nabla \f \varphi ) }-\intter{ \left \langle \f g ,\f \varphi\right \rangle } ={}&0 ,\quad
\end{split}\label{eq:velo}\\\begin{split}
\intter{( \t \f d, \f \psi(t) ) }  +\intter{\left (((\f v \cdot \nabla ) \f d, \f \psi)- \left (\sk{v}\f d , \f \psi\right ) + \lambda \left (\sy{v} \f d , \f \psi\right )+ \gamma\left\langle\f q, \f \psi\right\rangle\right )}={}&0
\end{split}
\label{eq:dir}
\end{align}%
\end{subequations}
for all test functions $\f \varphi\in L^{2}(0,T;\f W^{1,6}_{0,\sigma})$ and $\f \psi \in L^2(0,T;\f L^{3})$.
\end{defi}
The global existence of weak solutions was proved under the given assumptions~(\ref{eq:strong})--(\ref{Rand}) in~\cite[Theorem~3.1]{unsere} for a domain of class $\C^2$.

\begin{defi}\label{def:suit}
A weak solution $( \f v ,\f d )$ (see Definition~\ref{def:weak}) is said to be a \textit{suitable weak solution} to system~\eqref{eq:strong} if it is a weak solution and additionally satisfies the energy inequality
\begin{align}
\begin{split}
 \frac{1}{2}\|\f v( t) \|_{\Le}^2 &+ \F( \f d( t)) + \inttu{\mu_1\|\f d\cdot \sy{v}\f d\|_{L^2}^2+\mu_4 \|\sy{v}\|_{\Le}^2 } \\
& \quad +\inttu{
   ( \mu_5+\mu_6-\lambda(\mu_2+\mu_3))\|\sy{v}\f d\|_{\Le}^2+ \gamma \|\f q\|_{\Le}^2}\\
& \leq   \frac{1}{2}\|\f v_0 \|_{\Le}^2 + \F( \f d_0) + \inttu{  \langle \f g , \f v \rangle + ( \gamma ( \mu_2+ \mu_3) - \lambda ) \left ( \f q , \sy{v} \f d \right ) }\,
\end{split}
\label{energyin}
\end{align}
for almost all $t\in(0,T)$.
\end{defi}
\begin{defi}\label{def:strong}
A weak solution $( \vv, \dd)$  (see Definition~\ref{def:weak}) is said to be a \textit{strong solution} to~\eqref{eq:strong} if  it admits the additional regularity
\begin{align}
\begin{split}
\vv \in L^2 (0,T; \f W^{1,6} ) \, , \quad \dd \in L^\infty
(0,T ; \f L^{12} )\cap L^2(0,T; \f W^{2,3} )  \,,\quad
\dd \cdot \syv \dd \in L^2(0,T; \f L^6)\,.
\end{split}\label{strsol}
\end{align}
\end{defi}
\begin{rem}
For $\mu_1=0$ it would be sufficient to assume the regularity $\vv  \in L^2(0,T; \f W^{1,3}  \cap \f L^\infty  )$ and  $ \syv \dd \in L^2(0,T; \f L^3)$  instead of~$\vv \in L^2 (0,T; \f W^{1,6} )$ and $\dd \cdot \syv \dd \in L^2(0,T; \f L^6)$.
\end{rem}
We can now state the main theorem of this paper.
\begin{theorem}
\label{thm}
Let $\Omega\subset \R^3$ be a domain of class $\C^{2}$. Let $( \f v , \f d)$ be a suitable weak solution~(see Definition~\ref{def:suit}) to the Ericksen--Leslie system~(\ref{eq:strong})--(\ref{Rand}) and $( \vv,\dd)$ a strong solution (see Definition~\ref{def:strong}) to the same initial and boundary conditions~\eqref{boundary}--\eqref{Rand}.

Then
\begin{align*}
\f v \equiv \vv \, ,\quad \f d \equiv \dd \, .
\end{align*}
\end{theorem}
\begin{rem}
Theorem~\ref{thm} is a direct consequence of Lemma~\ref{lem2}. In Lemma~\ref{lem2}, even the continuous dependence on the initial values is shown as long as a strong solution exists.
\end{rem}
Before we present the proof of the main result, we give an important remark on the existence of suitable weak solutions.

\begin{rem}[Existence of suitable weak solutions]
In our recent work~\cite{unsere}, we proved global existence of weak solutions to the system~\eqref{eq:strong} in the sense of Definition~\ref{def:weak}. This is done by establishing a Galerkin approximation leading to an approximate system whose solutions $(\f v_n ,\f d_n)$ obey the energy equality as in~\eqref{energyeq}.
This allows us to show a priori estimates for the sequence of solutions to the approximate system and extract weakly- and weakly$^*$-converging subsequences.  In the end, it is possible to identify the limit of these subsequences with the solution $(\f v, \f d)$ to~\eqref{weak}. It turns out that the energy inequality~\eqref{energyin} cannot be shown to hold for the limit.
The a priori estimates for the approximate system~(see~\cite{unsere}) imply the following weak convergences
\begin{align*}
  \f v_{n }&\stackrel{*}{\rightharpoonup}  	\f v& \quad &\text{ in } L^{\infty} (0,T;\Ha)\cap L^{2} (0,T;\V)\cap W^{1,{2}}(0,T; ( \f W^{1,6}_{0,\sigma} )^*) \,,
  \\
\f q_n &\rightharpoonup  {\f q} &\quad &\text{ in }  L^{2} (0,T;\Le)\,,
\\
\syn v \f d_n &\rightharpoonup  \sy v \f d& \quad &\text{ in }  L^{2} (0,T;\Le)\,,
\\
\f d_n\cdot \syn v \f d_n &\rightharpoonup  \f d\cdot \sy v \f d &\quad &\text{ in }  L^{2} (0,T;L^2)\,,
\\
\f d_{n }&\stackrel{*}{\rightharpoonup}  		\f d &\quad& \text{ in } L^{\infty} (0,T;\He)\cap  L^{2} (0,T;\Hc)\cap W^{1,2}  (0,T;   \f  L^{{3/2}}  )\,.
\end{align*}
Due to the weakly lower semi-continuity of the appearing norms, one can deduce that
\begin{align}
\begin{split}\label{weaklow}
&\liminf_{n\ra \infty} \Big ( \frac{1}{2}\|\f v_n(t) \|_{\Le}^2 + \F( \f d_n(t)) + \inttu{\mu_1\|\f d_n\cdot \syn{v}\f d_n\|_{L^2}^2+\mu_4 \|\syn{v}\|_{\Le}^2 } \\
& \quad +\inttu{
   ( \mu_5+\mu_6-\lambda(\mu_2+\mu_3))\|\syn{v}\f d_n\|_{\Le}^2 +\gamma \|\f q_n\|_{\Le}^2 - \langle \f g , \f v _n\rangle }\Big ) \\
 &\geq \Big ( \frac{1}{2}\|\f v (t)\|_{\Le}^2 + \F( \f d( t)) + \inttu{\mu_1\|\f d\cdot \sy{v}\f d\|_{L^2}^2+ \mu_4 \|\sy{v}\|_{\Le}^2 } \\
& \quad +\inttu{
  ( \mu_5+\mu_6-\lambda(\mu_2+\mu_3))\|\sy{v}\f d\|_{\Le}^2 + \gamma \|\f q\|_{\Le}^2 - \langle \f g , \f v \rangle}\Big )
\, .
\end{split}
\end{align}
Note that $\f v_n \in \mathcal{C}_w([0,T]; \Ha)$ and $\f d_n \in \mathcal{C}_w ([0,T];\He)$.
However, we are not able to identify the limit of the remaining term $( \f q_n , \syn v \f d_n )$ since $\f q_n$ and $\syn v \f d _n $ only converge weakly. Thus, it is not clear weather a suitable solution in the sense of Definition~\ref{def:suit} exists.

Nevertheless, the existence of a suitable solution (see Definition~\ref{def:suit}) can be shown when assuming Parodi's relation $\gamma (\mu_2+\mu_3) = \lambda  $. Then the last term in the energy inequality~\eqref{energyin} vanishes, and with~\eqref{weaklow} the energy inequality also holds for the limit of the approximate solutions, which is the weak solution.
\end{rem}

\section{Properties of the strong solution}

\begin{lem}[Energy equality]\label{eneq}
A strong solution $(\vv,\dd) $ (see Definition~\ref{def:strong}) of the system~\eqref{eq:strong} fulfills the energy equality
\begin{align}
\begin{split}
 \frac{1}{2}\|\vv(t) \|_{\Le}^2 &+ \F( \dd(t)) + \inttu{\mu_1\|\dd\cdot \syv\dd\|_{L^2}^2+ \mu_4 \|\syv\|_{\Le}^2} \\
& \quad +\inttu{
   ( \mu_5+\mu_6-\lambda(\mu_2+\mu_3))\|\syv\dd\|_{\Le}^2 +\gamma \|\tilde{\f q}\|_{\Le}^2}\\
& =   \frac{1}{2}\|\vv_0 \|_{\Le}^2 + \F( \dd_0)  + \inttu{  \langle \f g , \vv \rangle + ( \gamma ( \mu_2+ \mu_3) - \lambda ) \left ( \tilde{\f q} , \sy{v} \f d \right ) }\,
\end{split}\label{energyeq}
\end{align}
for $t\in [0,T]$.

\end{lem}
\begin{proof}
Due to the regularity assumptions~\eqref{strsol} on the strong solution, we can take $( \vv , \tilde{\f q})$ as test functions  in~\eqref{weak} and obtain the energy equality in the same way as in~\cite[Proposition~4.1]{unsere}.
\end{proof}
\begin{lem}[Regularity of the strong solution]\label{reg}
A strong solution~$( \vv, \dd)$ (see Definition~\ref{def:strong}) admits the regularity
\begin{align}
\t \vv \in L^2 (0,T; (\V )^*)\, ,\quad
\t \dd \in L^1(0,T; \f L^3 )\, .
\label{regtilde}
\end{align}
\end{lem}
\begin{proof}

First, we estimate the time derivative of $\vv$.
Let $\f \varphi\in L^2(0,T;\V)$ be a test function in~\eqref{eq:velo}.
We estimate the terms individually. Because of the continuous embedding of $\V$ into $\f L^6$, we obtain for the convection term
\begin{align*}
\inttes{| ( ( \vv(t) \cdot \nabla ) \vv(t), \f \varphi(t) )|} &
\leq \| \vv \|_{L^\infty(\f L^2) } \| \nabla \vv \|_{L^2  ( \f L^3)}  \| \f \varphi \| _{L^2(\V)}\, .\\
\intertext{Similarly, the Ericksen stress can be estimated as}
\inttes{|( \nabla \dd(t) ^T \tilde{\f q}(t) , \f \varphi(t) )|}& \leq \| \nabla \dd \|_{L^\infty(\f L^2)} \| \tilde{\f q}\|_{L^2 ( \f L^3 )}  \| \f \varphi\|_{L^2(\V)}\,. \\
\intertext{For the right-hand side, we have that}
\inttes{|\langle \f g(t) , \f \varphi (t)\rangle |}&\leq \| \f g\|_{L^2 ( \Vd)}\| \f \varphi\|_{L^2(\Hb)}\, .
\end{align*}
With the definition of the Leslie stress tensor (see~\eqref{leslie}) , we get
\begin{align*}
\inttes{|( \f T^L(t); \nabla \f \varphi(t))  | }& \leq \Big( \mu_1 \| \dd \cdot \syv\dd  \|_{L^2(\f L^6)} \|\dd\|_{L^\infty(\f L^6)} ^2 + \mu_4 \| \syv \|_{L^2 (\f L^2)}\\ &\quad  +
( |\gamma(\mu_2+\mu_3) {}|+1 )\| \dd \|_{L^\infty ( \f L^6 )} \|  \tilde{\f q}\|_{L^2 (\f L^3)}   \\ & \quad + ((\mu_5+\mu_6) -\lambda(\mu_2+\mu_3))  \| \syv \|_{L^2( \f L^6)}\| \dd \|_{L^\infty ( \f L^6 )}^2 \Big ) \| \f \varphi \|_{L^2(\V) }\\
& \leq c\left ( \left (\| \dd\cdot \syv \dd  \|_{L^2(\f L^{6})} + \| \syv \|_{L^2( \f L^6)}+ \|  \tilde{\f q}\|_{L^\infty (\f L^3)}\right )\left (\| \dd \|_{L^\infty ( \f L^6 )}^2+1\right )\right )\| \f \varphi \|_{L^2(\V) }
\, .
\end{align*}
Due to the regularity assumptions on the strong solution (see Definition~\ref{def:strong}), the variational derivative of the free energy can be estimated in terms of the~$L^2(0,T;\f L^3)$-norm by standard embeddings and the Gagliardo--Nirenberg inequality~\cite[Section~21.19]{zeidler2A},
\begin{align}
\| \tilde{\f q} \|_{L^2(\f L^3)} \leq | \f \Lambda| \| \dd \|_{L^2(\f W^{2,3})} + \frac{1}{\varepsilon} \left ( \| \dd \|_{L^6(\f L^9)}^3 + \| \dd \|_{L^2(\f L^3)} \right ) \leq c \left (\| \dd \|_{L^2(\f W^{2,3})} +  \| \dd \|_{L^2(\f W^{2,3})}^{1/3} \| \dd \|_{L^\infty(\f L^6)}^{8/3}  +1\right ) \,.\label{qab}
\end{align}
Note that $\varepsilon$ is a constant parameter.
Altogether, we see that  $\t \vv \in L^2 (0,T;\Vd) $ and
\begin{align*}
\| \t \vv \|_{L^2(\Vd)}\leq {}& \| \vv \|_{L^\infty(\f L^2) } \| \nabla \vv \|_{L^2  ( \f L^3)} + \| \nabla \dd \|_{L^\infty(\f L^2)} \| \tilde{\f q}\|_{L^2 ( \f L^3 )}+\| \f g\|_{L^2 ( \Vd)}
\\& +c\left ( \left (\| \dd\cdot \syv \dd  \|_{L^2(\f L^{6})} + \| \syv \|_{L^2( \f L^6)}+ \|  \tilde{\f q}\|_{L^\infty (\f L^3)}\right )\left (\| \dd \|_{L^\infty ( \f L^6 )}^2+1\right )\right )
\end{align*}
with $\tilde{\f q}$ estimated in \eqref{qab}.

Recalling equation~\eqref{eq:dir} and estimate~\eqref{qab}, standard embeddings show that $\t \dd\in L^1(0,T;\f L^3 )$ with
\begin{align*}
\|\t \dd \|_{L^1(\f L^3)} &\leq \|\vv\|_{L^2(\f L^\infty)} \|\nabla \dd \|_{L^2(\f L^3)} + \|  \skv  \|_{L^2 (\f L^3)} \|\dd \|_{L^2( \f L^\infty)}  + | \lambda|\|\syv \|_{L^2 (\f L^3)}  \|\dd \|_{L^2( \f L^\infty)} +  \| \tilde{\f q}\|_{L^ 1 ( \f L^3)}\\
&\leq  c \left (\| \vv \|_{L^2(\f W^{1,6})} \|\dd\|_{L^2(\f W^{2,3})} + \| \tilde{\f q}\|_{L^ 2 ( \f L^3)}\right )\, .
\end{align*}
\end{proof}
In the course of the proof of Theorem~\ref{thm}, we shall employ the following three integration-by-parts formulae.
\section{Integration-by-parts formulae}
\begin{lem}\label{lem:intpart}
For functions $( \f v , \f d ) $ and $(\vv,\dd)$ fulfilling~\eqref{weakreg} and~\eqref{regtilde}, respectively, the integration-by-parts formulae
\begin{align}
\begin{split}
(\f v (t ) , \vv (t)) - ( \f v (s) , \vv (s)) &= \int_s^t \left (\langle \f v ( \tau ) , \t \vv ( \tau ) \rangle + \langle \t \f v ( \tau ) , \vv ( \tau )\rangle  \right )\text{\emph{d}} \tau\, ,\\
( \nabla \f d(t) ;\, \f \Lambda : \nabla \dd (t) ) - ( \nabla \f d (s);\, \f \Lambda : \nabla \dd (s)) & = -\int_s^t \left (( \t \f d(\tau) , \Delta_{\f \Lambda}\dd(\tau) ) + ( \Delta_{\f \Lambda} \f d(\tau), \t \dd(\tau) )\right ) \text{\emph{d}} \tau \, ,\\
( | \f d(t)|^2  , | \dd (t)|^2 ) - ( | \f d(s)|^2  , | \dd (s)|^2 )  &=2 \int_s^t \left (( \t \f d( \tau) \cdot \f d ( \tau), | \dd( \tau)|^2 ) + ( | \f d ( \tau)|^2  , \t \dd( \tau) \cdot \dd ( \tau))\right ) \text{\emph{d}} \tau \,
\end{split}\label{intpart}
\end{align}
hold true for every $s,t\in [0,T]$.
\end{lem}
\begin{proof}
We choose two approximate sequences $ \f v _n \in \C^1 ([0,T]; \V) $ and $ \vv_n \in \C^1 ([0,T]; \f W^{1,6}_{0,\sigma} )$ such that
\begin{align}
\begin{split}
\f v_n & \ra \f v \in % L^\infty(0,T; \Ha) \cap
L^2 (0,T; \V ) \cap W^{1,2} ( 0,T; ( \f W^{1,6}_{0,\sigma} )^*) \\
\vv _n & \ra \vv \in %L^\infty ( 0,T;\f L^\infty_\sigma ) \cap
 L^{2} ( 0,T; \f W^{1,6}_{0,\sigma} ) \cap W^{1,2} ( 0,T; \Vd) \, ,
 \end{split}\label{appr}
\end{align}
which is possible in view of density.
For the approximate sequences, the integration-by-parts formula
\begin{align}
(\f v_n (t ) , \vv_n (t)) - ( \f v_n (s) , \vv_n (s)) &= \int_s^t \left (( \f v_n ( \tau ) , \t \vv_n ( \tau ) ) + ( \t \f v_n ( \tau ) , \vv_n ( \tau ))\right ) \de \tau\, \label{intpartn}
\end{align}
obviously holds true for all $s,t\in[0,T]$.
In the following, we derive estimates for the terms on the left-hand side of~\eqref{intpartn}.
Let us define a partition of the unity via a function $\phi\in \C^1([0,T])$ with $$
  | \phi(t)|\leq 1 \quad \text{for all }t \in [0,T]\, , \quad  \phi(0)=0 \text{ and } \phi (T)= 1\,.$$
Let
$ \f u \in \C^1 ([0,T]; \V )$ and $\tilde{\f u}\in \C^1([0,T];\f W^{1,6}_{0,\sigma} ) $.
We have that
$$ \f u ( t) = \phi(t) \f u(t) + ( 1- \phi (t)) \f u(t)\,,\quad \tilde{\f u} ( t) = \phi(t) \tilde{\f u}(t) + ( 1- \phi (t)) \tilde{\f u}(t).$$
We abbreviate $\bar{\phi}(t):=  1-\phi(t) $ for all $t\in[0,T]$ such that $ \bar{\phi }(T)= 0$.  It can easily be seen that for $t\in[0,T]$
\begin{align*}
\phi(t) ( \f u(t) , \tilde{\f u}(t)) & = \phi(0) ( \f u(0) , \tilde{\f u}(0) )  + \int_0^t \phi'(\tau) ( \f u (\tau ) , \tilde{\f u}(\tau))\de \tau  + \int_0^t\phi(\tau )( ( \t \f u( \tau) , \tilde{\f u}(\tau)) + ( \f u(\tau), \t\tilde{\f u}(\tau) ) ) \de \tau \\
\bar{\phi }(t) ( \f u(t) , \tilde{\f u}(t)) &=  \bar{\phi }(T) ( \f u(T) , \tilde{\f u}(T) )  + \int_t^T \phi'(\tau) ( \f u (\tau ) , \tilde{\f u}(\tau)) \de \tau - \int_t^T \bar{\phi }(\tau )( ( \t \f u( \tau) , \tilde{\f u}(\tau)) + ( \f u(\tau), \t \tilde{\f u}(\tau) ) ) \de \tau \, .\\
\end{align*}
Summing up  the two previous equations, we  find that for all $t \in [0,T]$
\begin{align}
\begin{split}
( \f u(t) , \tilde{\f u}( t) ) &= \int_0^T \phi'(\tau ) ( \f u (\tau) , \tilde{\f u}(\tau)) \de \tau  +\int_0^T  \phi(\tau) ( ( \t \f u( \tau) , \tilde{\f u}(\tau)) + ( \f u(\tau),\t \tilde{\f u}(\tau) ) ) \de \tau \\&\quad - \int_t^T( ( \t \f u( \tau) , \tilde{\f u}(\tau)) + ( \f u(\tau), \t\tilde{\f u}(\tau) ) ) \de \tau \\
& \leq \max_{t\in[0,T]}| \phi'(t)| \| {\f u}\|_{L^2(\Le)} \| \tilde{\f u}\|_{L^2 ( \Le)}  + 2 \| \t \f u\|_{ L^{2}((\f W^{1,6}_{0,\sigma})^*)} \| \tilde{\f u}\|_{L^{2}(\f W^{1,6}_{0,\sigma} )}   + 2\| \f u\|_{L^2(\V)} \| \t \tilde{\f u}\|_{L^2 ( \Vd)}\, .
\end{split}\label{esti}
\end{align}
The above estimate is now applied to the left-hand side of~\eqref{intpartn}.
Since $\{\f v_n \}$ and $\{\vv_n\}$ are Cauchy sequences in the spaces indicated in~\eqref{appr}, we see with
$$ ( \f v_n(t), \vv_n(t)) - ( \f v_m(t), \vv_m(t)) =  ( \f v_n(t)-\f v_m(t),\vv_n(t)) + ( \f v _m(t) , \vv_n(t)-\vv_m(t))\quad \text{for } t\in [0,T]\quad \text{and }m,n\in\N $$
and the  estimate~\eqref{esti} that $\{( \f v_n, \vv_n)\}$ is a Cauchy sequence in $\C([0,T])$. The continuous functions are complete and the limit is unique such that $\{ ( \f v_n ,\vv_n)\} $ converges in $\C([0,T])$ to $( \f v , \vv)$.
For the approximation of the terms on the right-hand side of the identity~\eqref{intpartn}, we see that the difference of the approximation and the limit can be estimated by
\begin{align*}
&\left | \int_s^t \langle \f v ( \tau) , \t \vv ( \tau) \rangle - ( \f v_n ( \tau ) , \t  \vv_n ( \tau)) \de \tau \right | \leq \| \f v \|_{ L^2 ( \V)} \| \t \vv - \t \vv _n \|_{L^2( \Vd)}  + \| \f v - \f v _n \| _{ L^2 ( \Vd )}  \| \t \vv _n \| _ {L^2 ( \Vd)} \, ,\\
&\left | \int_s^t \langle \t  \f v ( \tau ) , \vv ( \tau) \rangle- ( \t \f v_n (\tau ) , \vv_n ( \tau)) \de \tau\right |  \leq \| \t \f v - \t \f v_n \| _{ L^{2}( ( \f W^{1,6}_{0,\sigma} )^*) } \| \vv \|_{L^{2}(\f  W^{1,6}_{0,\sigma})}   + \| \t \f v_n \| _{L^{2}(\f W^{1,6}_{0,\sigma} )^*) } \| \vv- \vv _n \|_{L^{2}(\f  W^{1,6}_{0,\sigma})} \, .
\end{align*}
The right-hand sides of the  above estimates converge to zero for $n\ra \infty$ since $\{\f v_n\}$ and $\{ \vv_n\}$ converge to $\f v $ and $\vv$ in the sense of~\eqref{appr}.
This proves that the integration-by-parts formula~\eqref{intpartn} holds for $\f v$ and $\vv$.

For proving the second and third formula in~\eqref{intpart}, the weak solution~$\f d $ and the strong solution $\dd$ are approximated by sequences of smooth functions $\f d_n $ and $\dd_n$ in $L^2(0,T;\Hc) \cap W^{1,2}(0,T;\f L^{3/2}) $ and $ L^2(0,T;\f W^{2,3})\cap W^{1,2}(0,T;\f L^2)$, respectively.
The approximate sequences can be chosen such that every element, i.e.,\,$\f d_n$ and $ \dd_n$,  of this sequences fulfills the same boundary conditions as $\f d$ and $\dd$,  respectively. Since these boundary values of $\f d_n$ and $\dd_n$ are constant in time, their time derivative vanishes on the boundary. Hence, there are no boundary terms in the integration-by-parts formula
\begin{align*}
\left ( \nabla \f d_n (t) ;\, \f \Lambda : \nabla \dd_n (t) \right ) - \left ( \nabla \f d_n (s) ;\, \f \Lambda : \nabla \dd_n (s) \right ) ={}& \int_s^t\left ( \left ( \nabla \t \f d_n ( \tau) ;\, \f \Lambda : \nabla \dd_n(\tau)\right ) + \left ( \nabla \f d_n(\tau) ;\, \f \Lambda : \nabla \t \dd_n(\tau)\right )\right ) \de \tau \\={}& -\int_s^t\left ( \left (  \t \f d_n ( \tau) , \Lap \dd_n(\tau)\right ) + \left ( \Lap \f d_n(\tau) , \t \dd_n(\tau)\right )\right ) \de \tau \,.
\end{align*}
Going to the limit in $n$ shows the second integration-by-parts formula in~\eqref{intpart}.
The third integration-by-parts formula in~\eqref{intpart} is proved by the same approximation.

\end{proof}
\section{Proof of the main result}
We define the relative energy for two solutions to system~\eqref{eq:strong}
\begin{align}
\E( \f v, \f d| \vv,\dd) : = \frac{1}{2}\| \f v- \vv\|_{\Le}^2 + \frac{1}{2}\| \nabla \f d- \nabla \dd \|_{ \f \Lambda }^2 + \frac{1}{4\varepsilon } \left \| ( |\f d|^2-1 )- ( | \dd|^2 -1)\right \|_{L^2}^2 \, ,\label{E}
\end{align}
and the relative dissipation by
\begin{align}
\begin{split}
\W( \f v, \f d| \vv,\dd) &:= \mu_1\left \| \f d\cdot \sy v \f d - \dd \cdot \syv \dd \right \| _ { L^2}^2 + \mu_4 \left \| \sy v - \syv \right \|_{\Le}^2 \\
  & \quad + ( \mu_5+\mu_6-\lambda(\mu_2+\mu_3))\left \|\sy{v}\f d- \syv \dd \right \|_{\Le}^2  + \gamma \left \|\f q-\tilde{\f q} \right \|_{\Le}^2
  \end{split}\label{W}
\end{align}
\begin{lem}
\label{lem2}
Let $(\f v, \f d)$ be a suitable weak solution (see Definition~\ref{def:suit}) to system~\eqref{eq:strong}  for given initial values $(\f v_0, \f d_0)$.
Let $( \vv, \dd)$ be a strong solution (see Definition~\ref{def:strong}) to system~\eqref{eq:strong} for given initial values $(\vv_0, \dd_0)$.
Then for almost all $t\in[0,T]$
\begin{align}
\E( \f v , \f d |\vv, \dd ) ( t) \leq \E(  \f v_0 , \f d_0| \vv_0, \dd_0 ) e^ {\int_0^t \mathcal{K}(\f v ,\f d |\vv,\dd)(s)\text{\emph{d}} s }\, ,\label{estima}
\end{align}
where $\mathcal{K}$
is given by
\begin{multline}
\mathcal{K}( \f v , \f d |\vv, \dd ) = \\ c \left (1+ \| \f d \|_{L^\infty(\f L^6)}^2 + \|\dd \|_{L^\infty(\f L^6)}^2\right ) \left (
\|\vv \|_{ \f W^{1,6}}^2+ \|\tilde{\f q}\|_{\f L^3}^2+\|\dd \cdot \syv \dd \|_{ \f L^6}^2+\| \t \dd \|_{\f L^3} + \| |\dd|^2-1\|_{L^6}^2+ \| \f v \|_{\f L^6}^2 + \| \nabla \dd\|_{\f L^2}^2  \right )
\,
\label{K}
\end{multline}
and $c$ is a possibly large constant.
\end{lem}
\begin{rem}\label{rem:regE}
The functional~\eqref{K} only depends on the two norms $\| \f v \|_{L^2(\f L^6)} $ and $\| \f d \|_{L^\infty ( \f L^6)}$ of the weak solution, which are known to be finite.
 Additionally, it depends on several norms of the strong solution $(\vv, \dd)$. Due to the regularity assumptions~\eqref{regtilde} and estimate~\eqref{qab}, the functional $\mathcal{K}$ is bounded in $L^1(0,T)$.
 For the relative energy and the relative dissipation, note that $\E(\f v ,\f d |\vv ,\dd)\in L^\infty(0,T)$ and $\W (\f v ,\f d |\vv ,\dd)\in L^1(0,T)$ due to~\eqref{weakreg} and~\eqref{energyin}, respectively.
 \end{rem}
\begin{proof}
Consider the relative energy
\begin{align*}
\E( \f v , \f d |\vv, \dd )   &= \frac{1}{2}\|\f v \|_{\Le}^2 + \frac{1}{2}\|\nabla \f d\|_{\f \Lambda}^2 + \frac{1}{4\varepsilon} \| |\f d|^2-1\|_{L^2}^2    + \frac{1}{2}\|\vv \|_{\Le}^2 + \frac{1}{2}\|\nabla \dd\|_{\f \Lambda}^2 + \frac{1}{4\varepsilon} \| |\dd|^2-1\|_{L^2}^2  \\
& \quad - ( \f v, \vv) - ( \nabla \f d;\,  \f \Lambda: \nabla \dd) - \frac{1}{2\varepsilon }( | \f d|^2-1, | \dd|^2 -1)\, .
\end{align*}
We insert the energy inequality~\eqref{energyin} for the weak solution $( \f v , \f d)$ and the energy equality~\eqref{energyeq} for the smooth solution $( \vv, \dd)$.
This leads to
\begin{align*}
\E( \f v , \f d |\vv, \dd )   &\leq  \frac{1}{2}\|\f v_0 \|_{\Le}^2+ \frac{1}{2}\|\vv_0 \|_{\Le}^2 + \F( \f d_0) + \F( \dd_0)\notag\\
 & \quad - \mu_1\intt{\|\f d\cdot \sy{v}\f d\|_{L^2}^2+\|\dd\cdot \syv\dd\|_{L^2}^2 } -
  \mu_4 \intt{\|\sy{v}\|_{\Le}^2+\|\syv\|_{\Le}^2 }\notag\\
  & \quad  -( \mu_5+\mu_6-\lambda(\mu_2+\mu_3)) \intt{\|\sy{v}\f d\|_{\Le}^2 +\|\syv\dd\|_{\Le}^2}-\gamma \intt{ \|\f q\|_{\Le}^2+ \|\tilde{\f q}\|_{\Le}^2 } \notag\\
& \quad + \intte{\langle \f g , \f v+\vv \rangle} +( \gamma ( \mu_2+ \mu_3) - \lambda ) \intt{\left ( \f q , \sy{v} \f d \right ) + \left (\tilde{\f q} , \syv \dd \right )}\notag\\
& \quad - ( \f v, \vv) - ( \nabla \f d; \,\f \Lambda : \nabla \dd) - \frac{1}{2\varepsilon }( | \f d|^2-1, | \dd|^2 -1)\, .
\end{align*}

Adding the integral over the relative dissipation gives
\begin{align}
\begin{split}
\E( \f v , \f d |\vv, \dd )   &+ \intte{\W( \f v , \f d |\vv, \dd )}\leq  \frac{1}{2}\|\f v_0 \|_{\Le}^2+ \frac{1}{2}\|\vv_0 \|_{\Le}^2 + \F( \f d_0) + \F( \dd_0)\\
 & \quad - 2\mu_1\intt{ \f d\cdot \sy{v}\f d,\dd\cdot \syv\dd }  -
  2\mu_4 \intt{\sy{v};\syv }\\
  & \quad  -2( \mu_5+\mu_6-\lambda(\mu_2+\mu_3)) \intt{\sy{v}\f d,\syv\dd}-2\gamma \intt{ \f q,\tilde{\f q} }\\
& \quad + \intte{\langle \f g , \f v+\vv \rangle}  +( \gamma ( \mu_2+ \mu_3) - \lambda ) \intt{\left ( \f q , \sy{v} \f d \right ) + \left (\tilde{\f q} , \syv \dd \right )}\\
& \quad - ( \f v, \vv) - ( \nabla \f d ;\,\f \Lambda: \nabla \dd) - \frac{1}{2\varepsilon }( | \f d|^2-1, | \dd|^2 -1)\, .
\end{split}\label{neg}
\end{align}
The last term can be written via Lemma~\ref{lem:intpart} as
\begin{align*}
-& \frac{1}{2\varepsilon}( | \f d|^2-1, | \dd|^2 -1)(t)+\frac{1}{2\varepsilon}( | \f d_0|^2-1, | \dd_0|^2 -1)  =    - \frac{1}{2\varepsilon }\intt{( |\f d|^2 -1 , \t | \dd|^2) + ( \t | \f d|^2 , |\dd|^2-1)}\\
&= - \frac{1}{2\varepsilon }\intt{( (|\f d|^2 -1)-(|\dd|^2-1) , \t | \dd|^2) + ( \t | \f d|^2+ \t |\dd|^2 , |\dd|^2-1)}\\
&=- \frac{1}{2\varepsilon }\intt{( (|\f d|^2 -1)-(|\dd|^2-1) , 2\t \dd \cdot ( \dd - \f d)) + ( \t | \f d-\dd|^2 , |\dd|^2-1)}\\
 & \quad - \frac{1}{2\varepsilon }\intt{( (|\f d|^2 -1)-(|\dd|^2-1) , 2\t \dd \cdot  \f d) + (2 \t (\f d \cdot \dd)  , |\dd|^2-1)}\\
 &=- \frac{1}{2\varepsilon }\intt{( (|\f d|^2 -1)-(|\dd|^2-1) , 2\t \dd \cdot ( \dd - \f d)) + ( \t | \f d-\dd|^2 , |\dd|^2-1)}\\
 & \quad - \frac{1}{\varepsilon }\intt{( (|\f d|^2 -1) \f d, \t \dd ) + ( \t \f d   , (|\dd|^2-1)\dd)}\\
\end{align*}
Recall the definition of the variational derivative of the free energy (see~\eqref{q}),
\begin{align*}
-& \frac{1}{\varepsilon }\intt{( (|\f d|^2 -1) \f d, \t \dd ) + ( \t \f d   , (|\dd|^2-1)\dd)} + \intt{(\t \dd , \Delta_{\f \Lambda} \f d )+ ( \t \f d , \Delta_{\f \Lambda} \dd)}
= - \intt{( \f q, \t \dd )+( \t \f d , \tilde{\f q})}\, .
\end{align*}
Due to the integration-by-parts formulae~\eqref{intpart} and the two previous equations, the last  line in~\eqref{neg} can be reformulated as
\begin{align}
\begin{split}
-& ( \f v, \vv) - ( \nabla \f d ;\,\f \Lambda:  \nabla \dd) - \frac{1}{2\varepsilon }( | \f d|^2-1, | \dd|^2 -1) = - ( \f v_0 , \vv_0) - ( \nabla \f d_0;\,\f \Lambda : \nabla \dd_0)- \frac{1}{2\varepsilon}( | \f d_0|^2-1, | \dd_0|^2 -1)  \\
&- \frac{1}{2\varepsilon }\intt{( (|\f d|^2 -1)-(|\dd|^2-1) , 2\t \dd \cdot ( \dd - \f d)) + ( \t | \f d-\dd|^2 , |\dd|^2-1)} \\
&
- \intt{\langle \f v , \t \vv \rangle+\langle\t \f v, \vv \rangle + ( \f q ,  \t \dd) + ( \t \f d , \tilde{\f q}) }\,.
\end{split}\label{abl}
\end{align}
Note that the sum of terms with the initial conditions $( \f v _0, \f d _0)$ and $( \vv_0,\dd_0)$ appearing in~\eqref{neg} and~\eqref{abl} is the relative energy of the initial values,
\begin{align*}
\frac{1}{2}\|\f v_0 \|_{\Le}^2+ \frac{1}{2}\|\vv_0 \|_{\Le}^2 + \F( \f d_0) + \F( \dd_0)  - ( \f v_0 , \vv_0) - ( \nabla \f d_0\,\f \Lambda  \nabla \dd_0) - \frac{1}{2\varepsilon}( | \f d_0|^2-1, | \dd_0|^2 -1)= \E( \f v_0 , \f d _0 | \vv_0,\dd_0)\,.
\end{align*}

We use the fact that $(\f v, \f d) $ and $( \vv, \dd)$ are solutions to calculate the last line in~\eqref{abl} explicitly.
In order to handle the last line in~\eqref{abl}, we use the fact that $( \f v ,\f d)$ and $( \vv, \dd)$ are solutions to~\eqref{eq:strong}. This shows with~\eqref{neg} that
\begin{align*}
\E&( \f v , \f d |\vv, \dd )   + \intte{\W( \f v , \f d |\vv, \dd )}\leq \E( \f v _0 , \f d_0| \vv_0,\dd_0)+\\ &
\intt{( (\vv\cdot \nabla ) \vv , \f v) +( ( \f v\cdot \nabla) \f v, \vv )}
+ \intt{ ( ( \vv \cdot \nabla )\dd , \f q) + ( ( \f v \cdot \nabla )\f d , \tilde{\f q}) - ( \nabla \dd^T\tilde{\f q }, \f v ) - (\nabla \f d^T \f q , \vv )}
\\&\quad
+ \mu_1 \intt{\left (\dd \cdot \syv \dd, \sy v : ( \dd \otimes \dd -\f d \otimes \f d)\right )+ \left (\f d \cdot \sy v \f d , \syv : ( \f d \o \f d - \dd \o \dd ) \right )  }
\\&\quad + ( \mu_5 + \mu_6 - \lambda ( \mu_2+ \mu_3) ) \intt{ \left (\syv \dd , \sy v (\dd- \f d ) \right )+\left (  \sy v \f d , \syv (\f d- \dd )  \right )}
\\&\quad
- \gamma(\mu_2+\mu_3) \intt{\left ( \dd \otimes \tilde{ \f q} ; \sy v- \syv\right )+\left ( \f d \otimes \f q ; \syv- \sy v \right )}\\&\quad
- \intt{\left  ( \dd \otimes \tilde{\f q }; \sk v \right  )+ \left ( \f d \otimes \f q ; \skv\right ) - (\skv \dd , \f q )- ( \sk v \f d , \tilde{\f q}) }\\&\quad
 + \lambda \intt{( \syv \dd , \f q- \tilde{\f q} )+ ( \sy v \f d , \tilde{\f q}- \f q)}
\\
&\quad - \frac{1}{2\varepsilon }\intt{( (|\f d|^2 -1)-(|\dd|^2-1) , 2\t \dd \cdot ( \dd - \f d)) + ( \t | \f d-\dd|^2 , |\dd|^2-1)}
\\&
= \E( \f v _0 , \f d_0| \vv_0,\dd_0) +I_1+ I_2+\mu_1 I_3+(\mu_5+\mu_6-\lambda( \mu_2+\mu_3)) I_4-\gamma (\mu_2+\mu_3) I_5  +I_6+\lambda I_7- \frac{1}{2\varepsilon }I_8\, .
\end{align*}
In the above equation we have employed the weak formulation for the solutions $( \f v , \f d )$ and $ ( \vv , \dd )$ tested with $(\vv, \tilde{\f q})$ and $(\f v , \f q)$, respectively.
Note that the choice of test functions is justified due to the additional regularity (see Lemma~\ref{reg}).
We calculate and estimate the terms in the above equation individually. In the following let $\delta>0$.  For the fist term $I_1$, we observe
that
 \begin{align*}
I_1 &= \intt{( ( \f v \cdot \nabla ( \f v - \vv ) , \vv - \f v ) + ( ( ( \f v - \vv )\cdot \nabla ) \vv , \vv- \f v ) }
 = \intte{\left ((\f v - \vv )\otimes ( \vv- \f v ) ; \syv\right )  }\\ &
  \leq C_\delta  \intte{ \|\syv\|_{\f L^{3}}^2 \|\f v - \vv \|^2_{\Le}} + \delta \intte{\| \f v - \vv\|_{\f L^6}^2}\, .
\end{align*}
We recall that $\f v $ and $\vv$ are solenoidal such that $( (\f v \cdot\nabla ) \f w , \f w) = ( ( \vv \cdot\nabla ) \f w ,\f w)=0$ for all $\f w \in \V$.

The  term $I_2$ can be estimated by
\begin{align*}
I_2 &= \intt{( ( \f v \cdot\nabla ) ( \f d - \dd), \tilde{\f q}) + ( ( \vv\cdot \nabla )( \dd - \f d ), \f q)} = \intt{( ( \f v - \vv ), ( \nabla \f d - \nabla \dd)^T \tilde{\f q}) + ( ( \vv\cdot \nabla )( \dd - \f d ) , \f q - \tilde{ \f q}) }\\
& \leq C_\delta    \intt{ \| \tilde{\f q}\|_{ \f L^{3}}^2\| \nabla \f d - \nabla \dd\|_{\f L^2}^2} + \delta \intte{\| \f v - \vv\|_{\f L^6}^2}  + C_\delta   \intte{\| \vv \|_{ \f L^{\infty}}^2\|\nabla \dd - \nabla \f d \|_{\f L^2}^2} + \delta \intte{\|\f q - \tilde{\f q }\|_{\Le}^2}\, .
\end{align*}
With the standard embedding $\V  \hookrightarrow \f L^6$  and Korn's inequality~\cite[Theorem~10.1]{mclean},  there is a constant $c$ such that $ \| \f v - \vv\|_{\f L^6}\leq c \| \sy v - \syv \|_{\Le}$.
We rearrange and estimate the term~$I_3$ by
\begin{align*}
I_3 & =
\intt{ \dd \cdot \syv \dd , \left (\sy v -\syv\right ): \left ( \dd \otimes \dd - \f d \otimes\f d\right )}
 \\& \quad  +\intt{\f d \cdot \sy v \f d - \dd \cdot \syv \dd , \syv: \left ( \f d \otimes \f d - \dd \otimes \dd \right ) }\\
 & = \intt{ \dd \cdot \syv \dd , \left (\sy v -\syv\right ): \left ( \dd \otimes (\dd - \f d) \right )}\\
&\quad + \intt{ \dd \cdot \syv \dd , \left (\sy v\f d -\syv\dd\right )\cdot  \left ( \dd - \f d \right )}\\
& \quad +  \intt{\dd \cdot \syv \dd, \syv: (\dd - \f d ) \otimes ( \dd - \f d )  }
 \\& \quad  +\intt{\f d \cdot \sy v \f d - \dd \cdot \syv \dd , \syv:  (\f d - \dd )\otimes \dd  }
 \\& \quad  +\intt{\f d \cdot \sy v \f d - \dd \cdot \syv \dd , \syv:  \f d \otimes (\f d - \dd)  }\\
 & \leq  C_\delta \| \dd \|_{L^\infty(\f L^6)}^2 \intte{ \|\dd \cdot \syv \dd\|_{ L^6}^2   \|\f d - \dd\|_{\f L^6}^2} + \delta\intte{\| \sy v- \syv\|_{\Le}^2}\\
 & \quad  +C_\delta  \intte{\|\dd \cdot \syv \dd\|_{ L^3}^2\|\f d - \dd\|_{\f L^6 }^2} +\delta\intte{\| \sy v\f d- \syv\dd\|_{\Le}^2}\\
 & \quad +    \intte{\left (\|\dd \cdot \syv \dd\|_{ L^3}^2+ \| \syv \|_{\f L^3}^2\right )\|\f d - \dd \|_{\f L^6}^2} \\
 & \quad  +C_\delta  \| \dd \|_{L^\infty(\f L^6)}^2 \intte{  \| \syv \|_{\f L^6}^2   \| \f d - \dd\|_{\f L^6}^2} + \delta\intte{\| \f d \cdot \sy v\f d- \dd \cdot \syv\dd\|_{L^2}^2}\\
& \quad + C_\delta  \| \f d \|^2_{L^\infty(\f L^6)} \intte{\| \syv \|^2_{\f L^6}\| \f d -\dd \|_{\f L^6}^2 }  + \delta\intte{\| \f d \cdot \sy v\f d- \dd \cdot \syv\dd\|_{L^2}^2}\,.
\end{align*}
The embedding $\Hb\hookrightarrow\f L^6$ together with Poincar\'{e}'s inequality (see~Morrey~\cite[Thm. 6.5.6.]{morrey}) assures that
$$ \| \f d -\dd \|_{\f L^6}  \leq c \| \nabla \f d - \nabla \dd\|_{\f \Lambda}\,.$$
 We continue with $I_4$,
 \begin{align*}
I_4&=  \intte{\left ( \syv \dd, \left (\sy v- \syv\right )( \dd -\f d)\right ) } + \intte{\left ( \sy v \f d- \syv \dd  , \syv ( \f d - \dd)\right ) }\\
&\leq C_\delta   \intte{\| \syv\dd  \|^2_{\f L^3} \|\f d - \dd \|_{\f L^6}^2} +\delta\intte{\| \sy v- \syv\|_{\Le}^2} \\
& \quad +C_\delta   \intte{\| \syv  \|^2_{\f L^3}\|\f d - \dd \|_{\f L^6}^2} +\delta\intte{\| \sy v\f d- \syv\dd \|_{\Le}^2}\,.
\end{align*}
 The term $I_5$ can be rearranged as
\begin{align*}
I_5 ={}& \intt{  \left ( \sy v - \syv , ( \dd - \f d ) \o \tilde{\f q} \right )  + \left ( \syv - \sy v , \f d \o ( \f q - \tilde{\f q} ) \right )  }\\
={}& \intt{\left ( \sy v -\syv , ( \dd - \f d ) \o \tilde{\f q } \right )+ \left ( \syv ( \f d - \dd ) , \f q -\tilde{\f q} \right )}\\
& + \intte{\left ( \syv \dd - \sy v \f d , \f q - \tilde{\f q}  \right )}\, ,
\end{align*}
and thus be estimated by
\begin{align*}
I_5 \leq{} & \delta\intte{\| \sy v- \syv\|_{\Le}^2}+\delta \intte{\|\f q - \tilde{\f q }\|_{\Le}^2}   +  C_\delta \intte{  \left ( \| \tilde{\f q }\|^2_{\f L^3}+\| \syv \|^2_{\f L^3} \right )\|\f d - \dd \|_{\f L^6}^2}
\\
& + \intte{\left ( \syv \dd - \sy v \f d , \f q - \tilde{\f q}  \right )}\, .
\end{align*}
The term $I_6$ is bounded by
\begin{align*}
I_6 &= \intt{\left ( \sk v ( \dd- \f d) , \tilde{ \f q} \right ) + \left ( \skv ( \f d - \dd) , \f q\right )} \\
&= \intt{\left (\left ( \sk v - \skv \right ) ( \dd- \f d) , \tilde{ \f q}\right ) + \left ( \skv ( \f d - \dd ) , \f q-\tilde{\f q}\right )}\\
& \leq   C_\delta  \intte{\| \tilde{\f q } \|^2_{\f L^3} \|\f d - \dd \|_{\f L^6}^2} +\delta\intte{\| \sk v- \skv\|_{\Le}^2} \\
& \quad + C_\delta   \intte{\| \skv \|^2_{\f L^3}\|\f d - \dd \|_{\f L^6}^2} + \delta \intte{\|\f q - \tilde{\f q }\|_{\Le}^2}\,.
\end{align*}
Note that due to Korn's~\cite[Theorem~10.1]{mclean} and Poincar\'{e}'s inequality, we find that
\begin{align*}
\| \sk v- \skv\|_{\Le} \leq  \| \f v- \vv\|_{\Hb}\leq c \|\sy v - \syv\|_{\Le} \, .
\end{align*}

The term $I_7$ is already in the  form desired.
Finally, we estimate $I_8$.  Starting  with the first term, we observe that
\begin{align*}
& \intt{\left ( (|\f d|^2 -1)-(|\dd|^2-1)\right ) ( \dd - \f d) , \t \dd  }
 \leq   \intte{\| \t \dd\|_{  \f L^3}\left ( \|(|\f d|^2 -1)-(|\dd|^2-1) \|_{L^2}^2 + \| \f d - \dd \|_{\f L^6}^2\right ) }\, .
\end{align*}
Since $\dd$ is a strong solution (see Definition~\ref{def:strong}),  $\t \dd $ is in $L^1(0,T;\f L^3)$ due to Lemma~\ref{reg}.

Now we reformulate the second term of $I_8$. Using that equation~\eqref{eq:dir} is fulfilled by $\f d$ and $\dd$, respectively, yields
\begin{align*}
\frac{1}{2}&\intte{\left (  | \dd|^2 -1, \t | \f d - \dd |^2  \right )} \notag \\
&= \intt{(|\dd|^2-1)(\f d -\dd), \t \f d - \t \dd}\notag \\
&= \intt{(|\dd|^2-1)(\dd -\f d), ( \f v \cdot \nabla ) \f d - ( \vv\cdot \nabla) \dd - \sk v \f d + \skv \dd }\\
& \quad + \intt{(|\dd|^2-1)(\dd -\f d),
  \lambda ( \sy v \f d - \syv \dd ) + \gamma ( \f q - \tilde{\f q }) }\\ & = J_{1}+J_{2}\,.
\end{align*}
The term $J_1$ can be rewritten as
\begin{align*}
J_{1}={}&\intt{(|\dd|^2-1)(\dd -\f d), \nabla \dd( \f v- \vv) } + \intte{(|\dd|^2-1, \f v \cdot (\nabla  \f d - \nabla \dd) ( \dd - \f d)) }\\ &   - \intt{(|\dd|^2-1)(\dd -\f d),   \sk v( \f d-\dd)} -\intt{(|\dd|^2-1)(\dd -\f d),
    \left ( \sk v -\skv\right ) \dd }\notag\,.
\end{align*}
 We  observe that the third term, i.e.,\,$ ((|\dd|^2-1)(\dd -\f d),   \sk v( \f d-\dd) )$ vanishes since $\sk v$ is skew-symmetric.  Hence, we can estimate $J_1$ by
\begin{align*}
J_1&\leq C_\delta   \intte{\left (\||\dd|^2-1\|_{L^6}^2+ \| \nabla \dd \|_{\f L^{2}}^2\right )\|\f d - \dd\|_{\f L^6}^2}  + \delta \intte{ \| \f v - \vv \|_{\f L^6}^2 }\notag\\
& \quad +  \intte{\left (\||\dd|^2-1\|_{L^6}^2+ \|\f v \|_{\f L^{6}}^2 \right )\left (\| \nabla \f d - \nabla \dd \|_{\f \Lambda}^2+ \| \f d - \dd \|_{\f L^6}^2\right )} \notag\\
& \quad + C_\delta \|\dd\|_{L^\infty(L^6)}^2  \intte{ \||\dd|^2-1\|_{L^6}^2  \| \f d - \dd \|_{\f L^6}^2} + \delta \intte{\| \sk v- \skv\|_{\Le}^2} \notag\, .
\end{align*}
It remains to estimate the term $J_2$:
\begin{align*}
J_{2}& \leq C_\delta   \|\dd\|_{L^\infty(L^6)}^2 \intte{\||\dd|^2-1\|^2_{L^6}\| \f d - \dd \|_{\f L^6}^2}  + \delta \intt{\lambda \| \sy v\f d - \syv\dd\|_{\Le}^2+\gamma \| \f q-\tilde{\f q}\|_{\Le}^2 }\, .
\end{align*}

Inserting everything back into~\eqref{neg} yields
\begin{align*}
\E( \f v , \f d |\vv, \dd )(t)& + \intte{\W( \f v , \f d |\vv, \dd )}  \leq  \E( \f v_0 , \f d _0 | \vv_0,\dd_0) \\&\quad +( \gamma ( \mu_2+ \mu_3) - \lambda ) \intte{\left ( \f q-\tilde{\f q} , \sy{v} \f d - \syv \dd \right )} \\
& \quad + \delta c \intte{\W( \f v , \f d |\vv, \dd )} +  \intte{\mathcal{K}( \f v , \f d |\vv, \dd ) \E( \f v , \f d |\vv, \dd )}\, .
\end{align*}
Since the constants are assumed to fulfill the dissipative relation~\eqref{con}
we can find a real number $\zeta \in (0,1)$ such that
\begin{align}
\left (\gamma(\mu_2+\mu_3)-\lambda\right )^2 \leq \zeta^2 4 \gamma ( \mu_5+\mu_6-\lambda(\mu_2+\mu_3)).
\end{align}
The relative energy can be estimated further on with Youngs and H\"{o}lder's inequality, such that
\begin{align*}
\E&( \f v , \f d |\vv, \dd )(t) + \intte{\W( \f v , \f d |\vv, \dd )}  \\ & \leq  \E( \f v_0 , \f d _0 | \vv_0,\dd_0) \\&\quad +  \zeta \intte{\left (\gamma \| \f q-\tilde{\f q}\|_{\Le}^2 + ( \mu_5+\mu_6-\lambda(\mu_2+\mu_3))\| \sy{v} \f d - \syv \dd \|_{\Le}^2\right )} \\
& \quad + \delta c \intte{\W( \f v , \f d |\vv, \dd )} +  \intte{\mathcal{K} ( \f v , \f d |\vv, \dd ) \E( \f v , \f d |\vv, \dd )}\\
& \leq \E( \f v_0 , \f d _0 | \vv_0,\dd_0) + (\zeta + \delta  c) \intte{\W( \f v , \f d |\vv, \dd )} +  \intte{\mathcal{K} ( \f v , \f d |\vv, \dd ) \E( \f v , \f d |\vv, \dd )}\,.
\end{align*}
We now choose $\delta $ sufficiently small such that $ \delta \leq (1- \zeta )/c$. Thus,  the relative dissipation $\W$ can be absorbed into the left-hand side.
Note that $c$ does not depend on $\delta$, but only on the constants arising from the embeddings and Korn's inequality as well as the constants of the system~(see~\eqref{con}).
The assertion~\eqref{estima} immediately follows from Gronwall's estimate (see Remark~\ref{rem:regE}).
\end{proof}
\begin{proof}[Proof of Theorem~\ref{thm}]
The main Theorem~\ref{thm} is a direct consequence of Lemma~\ref{lem2}.
\end{proof}
\bibliographystyle{abbrv}

\end{document}